\documentclass[journal,twoside,web]{ieeecolor}
 
\usepackage{lcsys}
\usepackage{generic}
\usepackage[colorlinks=true,linkcolor=blue, filecolor=blue, urlcolor=blue,citecolor=blue]{hyperref}
\usepackage{amsmath,amssymb,amsthm}
\usepackage{times}
\usepackage{mathptmx} 
\usepackage{textcomp}
\usepackage{graphicx}
\usepackage{epsfig}
\usepackage{caption}
\usepackage{subfig}   % if you are using \subfloat
\usepackage{float}
\usepackage{cite}
\usepackage[dvipsnames]{xcolor}
 
\usepackage{tikz}
\usepackage{pgfplots}
\pgfplotsset{compat = newest}
\usepackage{pgfplotstable}
\usetikzlibrary{positioning, arrows.meta}
\usepgfplotslibrary{fillbetween}
\newtheorem{theorem}{Theorem}[section]
\newtheorem{lemma}[theorem]{Lemma}

\newtheorem{corollary}[theorem]{Corollary}
\theoremstyle{definition}
\newtheorem{definition}{Definition}[section]
\newtheorem{assumption}[definition]{Assumption}

\theoremstyle{remark} 
\newtheorem{remark}{Remark}[section]
\newenvironment{mymatrix}
{\footnotesize\begin{bmatrix}}
	{\end{bmatrix}\normalsize}

\newcommand{\tb}[1]{\textcolor{black}{#1}}
\newcommand{\Z}{\mathrm{0}}
\newcommand{\I}{\mathrm{I}}

\DeclareMathOperator{\vech}{vech}

\newcommand{\col}{\operatorname{col}}
\newcommand{\row}{\operatorname{row}}
\newcommand{\blkdiag}{\operatorname{blkdiag}}
\newcommand{\vecop}{\operatorname{vec}}

\begin{document}

\def\BibTeX{{\rm B\kern-.05em{\sc i\kern-.025em b}\kern-.08em
    T\kern-.1667em\lower.7ex\hbox{E}\kern-.125emX}}
\markboth{\journalname, VOL. XX, NO. XX, XXXX 2017}
{Author \MakeLowercase{\textit{et al.}}: Preparation of Papers for IEEE Control Systems Letters (August 2022)}

\title{
Solution Sets for Inverse Infinite-Horizon Linear-Quadratic Descriptor Differential Games
}
 
 \author{Aaditya Kumar and Puduru Viswanadha Reddy, \IEEEmembership{Member, IEEE}
 	\thanks{Aaditya Kumar and P. V. Reddy are with the Department of Electrical Engineering, Indian Institute of Technology-Madras, Chennai, 600036, (e-mail: ee21d411@smail.iitm.ac.in, vishwa@ee.iitm.ac.in).}}

\maketitle
\thispagestyle{empty}

%%%%%%%%%%%%%%%%%%%%%%%%%%%%%%%%%%%%%%%%%%%%%%%%%%%%%%%%%%%%%%%%%%%%%%%%%%%%%%%%
\begin{abstract} 
In this letter, we study a model-based inverse problem for infinite-horizon linear-quadratic differential games with descriptor dynamics. Given an observed feedback strategy profile, we seek to identify all cost functions that rationalize it as a feedback Nash equilibrium; this collection is referred to as the solution set. We characterize the solution set, show that it is rectangular and convex, and provide an algorithm for computing an admissible realization whenever it is nonempty. \tb{We also show that, compared	with the corresponding inverse problem for standard state-space dynamics, descriptor dynamics modify the geometry of the solution set and may reduce identifiability.} Finally, we illustrate the results with numerical examples.
\end{abstract}

\begin{IEEEkeywords}
Game theory; inverse differential games; descriptor systems; feedback Nash equilibrium 
\end{IEEEkeywords}

\section{Introduction}
\label{sec:introduction} 
\tb{Dynamic game theory provides a rigorous framework for multi-agent strategic
 interactions evolving over time, with foundational treatments in
 \cite{Basar:1999,Engwerda:2005}. Advances in data-driven methods have spurred
 growing interest in  {inverse dynamic games} (IDGs), which seek to infer
 agents' objectives from observed behavior. IDGs have found application in
 human--machine shared control \cite{IngaECC2021,WuWang2024}, autonomous
 driving \cite{LeCleach2021,Armstrong2024}, biological motion
 \cite{Molloy2018Birds}, multi-agent robotics \cite{Mehr:23}, and electricity
 markets \cite{Risanger2020}.} 
 
\tb{Inverse problems in the single-agent setting have been studied extensively
under the inverse optimal control (IOC) and inverse reinforcement learning (IRL)
\cite{Kalman:1964,ab2020inverse} frameworks. In the multi-agent setting,
identification may be model-based (known dynamics) or model-free (dynamics and
objectives jointly estimated).
For finite-horizon games, optimality conditions from the Pontryagin maximum
principle yield identification schemes \cite{Molloy2017IFACP,Molloy2020TAC}.
For infinite-horizon linear--quadratic differential games (LQDGs) under
feedback information structures, IOC-based methods
\cite{Rothfuss2017IFAC}, solution-set characterizations via coupled algebraic
Riccati equations \cite{Inga2019CSL}, and a frequency-domain approach
\cite{Huang:2022} have been proposed. Model-based IRL and model-free
identification methods for this setting were later developed in
\cite{Martirosyan:2023,Martirosyan2024LCSS}.}
 
\tb{Most existing IDG formulations assume that agents' interactions are governed by differential or difference equations. However, many multi-agent systems are subject to instantaneous algebraic constraints arising from physical interconnections or conservation laws, as in power networks~\cite{Liu:20} and constrained robotic systems~\cite{Peng:20}, and therefore such systems   require modeling with  differential-algebraic equations (DAE). This distinction is important in inverse problems because modeling constraint-driven behavior with unconstrained dynamics, or misspecifying the interaction structure, may incorrectly attribute part of the observed behavior to strategic preferences rather than to physical constraints, yielding misinterpreted costs. A DAE-based IDG framework explicitly separates constraint-induced behavior from true strategic behavior, ensuring that the recovered objectives reflect the underlying interaction structure. While descriptor systems have been studied in forward linear--quadratic Nash games~\cite{Engwerda:12,Reddy:13,Tanwani:20}, inverse formulations for such systems have not been developed, to the best of our knowledge.} 

\tb{Motivated by this observation, we study a model-based inverse problem for infinite-horizon linear--quadratic descriptor differential games (LQDDGs) under a feedback information structure. The main contributions of this letter are threefold. First, we formulate the corresponding inverse dynamic game: given an observed stabilizing feedback strategy profile, we characterize all cost parameters for which it is a feedback Nash equilibrium (FBNE). To this end, we build on known conditions for FBNE in LQDDGs~\cite{Engwerda:12,Reddy:13} and on structural ideas from inverse LQDG formulations~\cite{Inga2019CSL}. Second, we characterize the associated solution set and show that it is convex and admits a Cartesian-product decomposition across players, as in~\cite{Inga2019CSL}. In particular, we show that descriptor dynamics modify the geometry of the inverse problem. As a result, the solution set may be enlarged, reducing identifiability relative to the standard inverse LQDG case~\cite{Inga2019CSL}. Third, we provide a convex optimization approach for computing an admissible realization whenever the solution set is nonempty.}
 
This letter is organized as follows. In Section \ref{sec:preliminaries}, we  provide preliminaries and inverse problem formulation. In Section \ref{sec:MainResults}, we characterize the solution set associated with the inverse problem and analyze its structure. In Section \ref{sec:Numerical}, we illustrate our approach for identification with numerical simulations, and finally Section \ref{sec:Conclusions} concludes.
 
 \subsubsection*{Notation}  Let $\mathbb{R}^n$ and $\mathbb{R}^{n\times m}$ denote the sets of real $n$-vectors and $n\times m$ real matrices, respectively; $\mathbb{S}^n$ the set of $n\times n$ real symmetric matrices; and $\mathbb{C}^-$ the open left half-plane. For a matrix $A$, $A^\top$, $\ker(A)$, $\mathrm{rank}(A)$, and $\sigma(A)$ denote its transpose, kernel, rank, and spectrum, respectively; $A$ is stable if $\sigma(A)\subset\mathbb{C}^-$. We write $\I_n$ for the $n\times n$ identity matrix and $\Z$ for the zero matrix of appropriate dimensions; for $A\in\mathbb{S}^n$, $A\succ0$ denotes positive definiteness. For an $N$-tuple $(A_1,\dots,A_N)$, we write $(A_i,A_{-i})$, where $A_{-i}:=(A_1,\dots,A_{i-1},A_{i+1},\dots,A_N)$. The Kronecker and Cartesian products are denoted by $\otimes$ and $\prod$, respectively. The operators $\col\{A_i\}_{i=1}^N$, $\row\{A_i\}_{i=1}^N$, and $\blkdiag\{A_i\}_{i=1}^N$ denote the block column, block row, and block diagonal matrices formed from $\{A_i\}_{i=1}^N$; we use $A:=(A_i,A_{-i})$ and $A:=\col\{A_i\}_{i=1}^N$ interchangeably where convenient. For $A\in\mathbb{R}^{n\times m}$, $\mathrm{vec}(A)\in\mathbb{R}^{nm}$ denotes the column-wise vectorization, while for $A\in\mathbb{S}^n$, $\vech(A)\in\mathbb{R}^{n(n+1)/2}$ denotes the half-vectorization obtained by stacking the lower-triangular entries column-wise. Moreover, for $A\in\mathbb{S}^n$, $\mathrm{vec}(A)=\mathcal D_n\vech(A)$, where $\mathcal D_n$ is the $n^2\times n(n+1)/2$ duplication matrix \cite[Chapter~3]{magnus2019matrix}. The Euclidean norm of $x\in\mathbb{R}^n$ is denoted by $\|x\|_2$.
\section{Preliminaries and Problem Formulation}
\label{sec:preliminaries}
In this section, we present background on linear--quadratic descriptor differential games (LQDDGs) and introduce the problem formulation.
\subsection{LQDDG and Feedback Nash Equilibrium}
We consider an $N$-player infinite-horizon nonzero-sum LQDDG. 
Let $\mathsf{N} := \{1,2,\ldots,N\}$ denote the set of players. 
At each time instant $t \in [0,\infty)$, every player $i \in \mathsf{N}$ chooses a control input 
$u_i(t) \in \mathbb{R}^{m_i}$ and influences the evolution of the state variable $x(t) \in \mathbb{R}^n$ according to the following linear differential--algebraic equation (DAE)
\begin{align}
	E\dot{x}(t) &= A x(t) + \sum_{i \in \mathsf{N}} B_i u_i(t),
	\qquad x(0) = x_0,
	\label{eq:state_eqn}
\end{align}
where $E \in \mathbb{R}^{n \times n}$ with $\operatorname{rank}(E)=r < n$, 
$A \in \mathbb{R}^{n \times n}$, 
$B_i \in \mathbb{R}^{n \times m_i}$, 
and $x_0 \in \mathbb{R}^n$.  
Each player $i \in \mathsf{N}$ seeks a control strategy $u_i$ that minimizes the quadratic cost functional
\begin{align}
	\tb{J_i(u_i,u_{-i};x_0)} 
	= \int_0^\infty 
	\Big(
	x^\top(t) Q_i x(t)
	+ \sum_{j \in \mathsf{N}} 
	u_j^\top(t) R_{ij} u_j(t)
	\Big) dt,
	\label{eq:cost_functional}
\end{align}
where $Q_i \in \mathbb{S}^n$ and $R_{ij} \in \mathbb{S}^{m_j}$ for $j\in \mathsf N$.

An initial state $x_0$ in \eqref{eq:state_eqn} is called consistent if the DAE admits a solution for that initial condition; see \cite{Kalogeropoulos:1998}. The DAE \eqref{eq:state_eqn}$,$ or equivalently the pair $(E,A)$, is called regular if $\det(\lambda E-A)\neq 0$ for some $\lambda\in\mathbb{C}$. It admits a unique solution for every consistent initial state if and only if $(E,A)$ is regular; see \cite{Brenan:96}. Moreover, if $(E,A)$ is regular, then by the Kronecker canonical form (KCF) \cite{Gantmacher:00}, there exist nonsingular matrices $X,Y\in\mathbb{R}^{n\times n}$ such that
\begin{align}
	Y^\top E X &=
	\begin{bmatrix}
		\I_r & \Z \\
		\Z & N
	\end{bmatrix},
	\qquad
	Y^\top A X =
	\begin{bmatrix}
		J & 0 \\
		0 & \I_{n-r}
	\end{bmatrix},
	\label{eq:KCF}
\end{align}
where $J$ is the $r\times r$ Jordan matrix associated with the finite eigenvalues of the pencil $\lambda E-A$, and $N$ is a nilpotent Jordan matrix of size $(n-r)\times(n-r)$ satisfying $N^\mu=0$ and $N^{\mu-1}\neq 0$. The integer $\mu$ is called the index of $(E,A)$. In general, the solution of an index-$\mu$ DAE depends on derivatives of the control input up to order $\mu-1$. To avoid differentiating the inputs, we restrict attention to index-1, or impulse-free, DAEs, as formalized in the next assumption.
\begin{assumption}\label{ass:index1} 
	The pair $(E,A)$ is regular and has index $1$. 
\end{assumption}

Using the KCF \eqref{eq:KCF}, define the coordinate transformation
\begin{align}
	\begin{bmatrix}x_1(t) \\ x_2(t)\end{bmatrix} := X^{-1}x(t),\label{eq:transformation}
\end{align}
where $x_1(t)\in\mathbb{R}^r$ and $x_2(t)\in\mathbb{R}^{n-r}$.
Partition $X$ as $X=[X_1~X_2]$, where $X_1\in\mathbb{R}^{n\times r}$ and $X_2\in\mathbb{R}^{n\times(n-r)}$.
Define $\bar B_{1i}:=[\I_r~0]Y^\top B_i\in \mathbb R^{r\times m_i}$, $\bar B_{2i}:=[0~\I_{n-r}]Y^\top B_i\in \mathbb R^{(n-r)\times m_i}$, and let $B:=\row\{B_i\}_{i=1}^N\in \mathbb R^{n\times m}$, where $m=\sum_{i\in\mathsf N} m_i$.
The DAE \eqref{eq:state_eqn} then decomposes into dynamic and algebraic parts as
\begin{subequations}
	\begin{align}
		x(t) &= X_1 x_1(t) + X_2 x_2(t), \label{eq:DAE-decompose}\\
		\dot{x}_1(t) &= J x_1(t) + \sum_{i\in\mathsf N} \bar B_{1i} u_i(t), \label{eq:DAE-DE1}\\
		x_2(t) &= -\sum_{i\in\mathsf N} \bar B_{2i} u_i(t), \quad
		\begin{bmatrix}x_1^\top(0)&x_2^\top(0)\end{bmatrix}^\top=X^{-1}x_0. \label{eq:DAE-DE2}
	\end{align}
	\label{eq:DAE-DE}
\end{subequations}
\tb{  
Hence, the descriptor system~\eqref{eq:state_eqn} has only $r$ differential degrees of freedom: $x_1$ evolves according to~\eqref{eq:DAE-DE1}, whereas $x_2$ is determined algebraically by the inputs through~\eqref{eq:DAE-DE2}.}

We assume that players use linear static state-feedback strategies, i.e.,
$u_i(t)=F_i x(t)$, where $F_i\in\mathbb{R}^{m_i\times n}$ for each $i\in\mathsf N$.
Let $u(t)=\col\{u_i(t)\}_{i=1}^N$ and $F:=\col\{F_i\}_{i=1}^N$.
The system~\eqref{eq:state_eqn} is said to be finite-dynamics stabilizable if there exists a feedback $u(t)=Fx(t)$ such that all finite eigenvalues of the closed-loop system
$E\dot{x}(t)=(A+BF)x(t)$ are stable.
This property holds if and only if $\operatorname{rank}[\lambda E-A\;\;B]=n$ for all $\lambda\in\mathbb C^+$; see~\cite{Engwerda:12}.
We therefore impose the following assumption.
\begin{assumption}\label{ass:stab}
The pairs $(J,\bar B_{1i})$ are stabilizable for all $i\in\mathsf N$, so that each player can individually stabilize the finite dynamics. Accordingly, we restrict our analysis to
	\begin{align}
		\mathsf F_s&:=\big\{F\in\mathbb R^{m\times n}\;\big|\;(E,A+BF)\text{ is regular, has index }1,\notag\\
		&\text{and all finite eigenvalues of }(E,A+BF)\text{ lie in }\mathbb C^-\big\}.
	\end{align}
\end{assumption}
\tb{Assumption~\ref{ass:stab} serves two purposes. First, it ensures that the closed-loop descriptor system is impulse-free. Second, it guarantees that the costs in~\eqref{eq:cost_functional} are finite; see also \cite{Engwerda:12,Reddy:13}.} 
We now introduce the notion of a feedback Nash equilibrium for the LQDDG \eqref{eq:state_eqn}-\eqref{eq:cost_functional}.

\begin{definition}
	A strategy profile $F^\star:=(F_i^\star,F_{-i}^\star)\in \mathsf F_s$ is a feedback Nash equilibrium (FBNE) for the LQDDG \eqref{eq:state_eqn}-\eqref{eq:cost_functional} if, for each $i\in \mathsf N$,
	\begin{align}
		\tb{J_i((F_i^\star,F_{-i}^\star);x_0)}
		\leq
		\tb{J_i((F_i,F_{-i}^\star);x_0)},
		\label{eq:FBNEcond}
	\end{align}
	for all $F_i\in \mathbb R^{m_i\times n}$ such that $(F_i,F_{-i}^\star)\in \mathsf F_s$, and for all consistent initial conditions $x_0$. 	
	We denote the set of all FBNE associated with the cost functions \eqref{eq:cost_functional} by
	\begin{align}
		\mathcal F(\{Q_i,R_{ij}\}_{i,j\in \mathsf N})
		:=
		\big\{F\in \mathsf F_s~|~ F \text{ satisfies } \eqref{eq:FBNEcond} \big\}.
		\label{eq:FBNEset}
	\end{align}
\end{definition}

\subsection{Problem formulation} 
The model-based inverse problem studied in this paper is formulated as follows.
\subsubsection{Problem}  \label{prob:prob1}
Assume that the system matrices $E$, $A$, and $B_i$, $i\in\mathsf N$, in~\eqref{eq:state_eqn} are known and satisfy Assumptions~\ref{ass:index1} and~\ref{ass:stab}. Let an observed feedback strategy profile $F=(F_i,F_{-i})\in\mathsf F_s$ be given. Under the quadratic cost structure~\eqref{eq:cost_functional}, characterize the set of cost parameters $\{Q_i,R_{ij}\}_{i,j\in\mathsf N}$, with $Q_i\in\mathbb S^n$ and $R_{ij}\in\mathbb R^{m_i\times m_j}$, for which $F$ is a feedback Nash equilibrium of the LQDDG. 
Equivalently, determine the inverse solution set
\begin{align}
\Theta(F):=
\left\{
\{Q_i,R_{ij}\}_{i,j\in N}
\;\middle|\;
F\in\mathcal{F} (\{Q_i,R_{ij}\}_{i,j\in N})
\right\}.
	\label{eq:solset1}
\end{align}
\begin{remark}\label{rem:feedbackform}
	The inverse problem \ref{prob:prob1} can also be formulated in terms of observed state and control trajectories $(x,u_i,u_{-i})$ rather than feedback matrices $(F_i,F_{-i})$, since the latter can be recovered by least squares; see~\cite{Inga2019CSL,Martirosyan:2023}.
\end{remark}
%%%%%%%%%%%%%%
%%%%%%%%%%%%%%
\section{Inverse LQDDG}
\label{sec:MainResults}
In this section, we first recall necessary and sufficient
conditions for the existence of a feedback Nash equilibrium
for the LQDDG. We then reformulate these conditions to
characterize the solution set \eqref{eq:solset1} of the inverse problem. 
\subsection{FBNE -- Necessary and sufficient conditions} 
We introduce notation used in the remainder of the paper.
Denote $\bar B_1:=\row\{\bar B_{1i}\}_{i=1}^N$, $\bar B_2:=\row\{\bar B_{2i}\}_{i=1}^N$, and $R_i:=\blkdiag\{R_{ij}\}_{j=1}^N$. 
Using \eqref{eq:DAE-DE2} in \eqref{eq:cost_functional}, the LQDDG \eqref{eq:state_eqn}--\eqref{eq:cost_functional} reduces to a standard LQDG with $r$-dimensional state dynamics:
\begin{subequations}
	\begin{align}
		\dot{x}_1(t)&=Jx_1(t)+\bar B_1 u(t), 
		\quad x_1(0)=\begin{bmatrix}\I_r&0\end{bmatrix} X^{-1}x_0, \label{eq:reddyn}
	\end{align}
	and the objective  of each player $i\in\mathsf N$ as
	\begin{align}
		\tb{J_i(u;x_0)}=
		\int_0^\infty 
		\begin{bmatrix} x_1^\top (t)&u^\top(t)\end{bmatrix}
		M_i
		\begin{bmatrix} x_1^\top(t)&u^\top(t) \end{bmatrix}^\top 
		dt, \label{eq:redobj}
	\end{align}
where
\begin{align} 
&M_i=	\begin{bmatrix} X_1&-X_2 \bar B_2 \\0&\I_{n-r}\end{bmatrix}^\top \begin{bmatrix}Q_i & 0\\0& R_i \end{bmatrix} 
\begin{bmatrix} X_1&-X_2 \bar B_2 \\0&\I_{n-r}\end{bmatrix}\notag 
\\
&\qquad \qquad =: \begin{bmatrix}
	\bar Q_{i} &   \bar V_{i1} &  \bar V_{i2} & \cdots &   \bar V_{iN}\\
	\bar V_{i1}^\top &   \bar R_{i1} &   \bar S_{i12} & \cdots &   \bar S_{i1N}\\
	\bar V_{i2}^\top &   \bar S_{i12}^\top &   \bar R_{i2} & \cdots &   \bar S_{i2N}\\
	\vdots & \vdots & \vdots & \ddots &\vdots\\
	\bar V_{iN}^\top &     \bar S_{i1N}^\top & \bar S_{i2N}^\top  & \cdots  &   \bar R_{iN} 
\end{bmatrix},\\
&\bar Q_i = X_1^\top Q_i X_1 \in \mathbb R^{r\times r},\\
&\bar R_{ij}=R_{ij}+\bar B_{2j}^\top X_2^\top Q_i X_2 \bar B_{2j}\in \mathbb R^{m_j \times m_j},~j\in \mathsf N,\\
&\bar V_{ij}=-X_1^\top Q_i X_2 \bar B_{2j}\in \mathbb R^{r\times m_j},~j\in \mathsf N,\\
&\bar S_{ijk}=\bar B_{2j}^\top X_2^\top Q_i X_2 \bar B_{2k} \in \mathbb R^{m_j \times m_k},\\
&\qquad \qquad k=j + 1,\cdots,N,~j=\mathsf N\backslash \{N\}. \notag 
\end{align}
\end{subequations}
 \tb{Under Assumptions \ref{ass:index1} and \ref{ass:stab}, the feedback class $\mathsf F_s$ is index preserving, i.e., $\text{index}(E,A)=\text{index}(E,A+BF)$ for any $F\in \mathsf F_s$. Properties of such feedback classes were studied in \cite{Reddy:13}. The next result shows how a full-state feedback law for \eqref{eq:state_eqn} induces a feedback law for the reduced system \eqref{eq:reddyn}. Related results were established in \cite[Theorems 3.6, 4.2--4.4]{Reddy:13} using matrix projector theory. Here, we recall the result most relevant to our setting and provide an alternative proof based on the Kronecker canonical form, which will be useful later in characterizing the solution sets of the inverse problem.}
 \begin{lemma}
 	\label{lem:highlowtransformation}
 	Let Assumptions \ref{ass:index1} and \ref{ass:stab} hold. Then there exists a map 
 	$\Omega:\mathsf F_s \subset \mathbb R^{m\times n}\rightarrow \mathbb R^{m\times r}$ such that for any $F\in \mathsf F_s$,
 	\begin{subequations} 
 		\begin{align}
 			u(t) &= Fx(t)=\Omega(F)x_1(t), \label{eq:infnonunique1}\\
 			\Omega(F) &:=F(\I_{n-r}+X_2\bar B_2 F)^{-1}X_1. \label{eq:infnonunique2}
 		\end{align}
 		Moreover, $\sigma(J+\bar B_1\Omega(F))\subset \mathbb C^-$ and coincides with the finite spectrum of the pair $(E,A+BF)$. 
 		Further, the preimage of any $\bar F\in \mathrm{Im}(\Omega)$ is given by
 		\begin{align}
 			\Omega^{-1}(\bar F)
 			&:=\big \{F\in \mathsf F_s ~\big|~FS=\bar F,\notag\\
 			&\qquad \text{where } S=(X_1-X_2\bar B_2 \bar F)\in \mathbb R^{n\times r}\big\}.
 			\label{eq:preimageOmega}
 		\end{align}
 	\end{subequations} 
 \end{lemma}
 \begin{proof}
 	\tb{	Using $u(t)=Fx(t)$ in \eqref{eq:DAE-DE2} gives
 		\[
 		x_2(t)=-\bar B_2Fx(t)
 		=-\bar B_2FX_1x_1(t)-\bar B_2FX_2x_2(t).
 		\]
 		From \cite[Lemmas 7, 8]{Engwerda:12}, as the pair $(E,A+BF)$ has index 1, it follows that the matrix $(\I_{n-r}+\bar B_2FX_2)$ is invertible, and the following relation holds
 		$$(\I_{n-r}+\bar B_2FX_2)^{-1}\bar B_2 F=\bar B_2 F(\I_n+X_2\bar B_2F)^{-1}.$$
 		Using this, we get
 		$$ x_2(t)=-(\I_{n-r}+\bar B_2FX_2)^{-1}\bar B_2FX_1x_1(t).$$
 		Substituting into $u(t)=Fx(t)=F(X_1 x_1(t)+ X_2 x_2(t))$ yields
 		\begin{align*} 
 			u(t)	&=F\left(\I_n-X_2 (\I_{n-r}+\bar B_2 F X_2)^{-1}\bar B_2 F \right)X_1 x_1(t)\\
 			&=F\left(\I_n-X_2 \bar B_2F (\I_n+ X_2 \bar B_2 F )^{-1}  \right)X_1 x_1(t)\\
 			&=F(\I_n+X_2 \bar B_2 F)^{-1}X_1 x_1(t)=\Omega(F)x_1(t).
 		\end{align*} 
 		From the KCF of the index-$1$ regular pair $(E,A)$,
 		\[
 		Y^\top(\lambda E-(A+BF))X=
 		\begin{bmatrix}
 			\lambda \I-(J+\bar B_1FX_1) & -\bar B_1FX_2\\
 			-\bar B_2FX_1 & -(\I+\bar B_2FX_2)
 		\end{bmatrix}.
 		\]
 		Since $(I+\bar B_2FX_2)$ is invertible, Schur's determinant formula gives
 		$
 		\det(Y^\top(\lambda E-(A+BF))X)
 		=\det\!\left(\lambda \I_r-(J+\bar B_1\Omega(F))\right)$. 
 		Since $F\in\mathsf F_s$, the roots of $\det(\lambda E-(A+BF))=0$ lie in $\mathbb C^-$. Hence $J+\bar B_1\Omega(F)$ is stable.}
 	
 	\tb{Finally, for $\bar F\in\mathrm{Im}(\Omega)$ there exists $F\in\mathsf F_s$ such that 
 		$
 		\bar F = F(\I_n+X_2 \bar B_2 F)^{-1}X_1
 		=F(X_1-X_2 \bar B_2 F(\I_n+X_2 \bar B_2 F)^{-1}X_1)
 		=F(X_1 - X_2 \bar B_2 \bar F)=FS$.} 
 \end{proof}
%%%%%%%%%%%%%%%%%%%%%%%%%%%%%%%%%%%%%%%%
The next theorem gives necessary and sufficient conditions for the existence of an FBNE for the LQDDG; see \cite{Reddy:13,Engwerda:12}. 
\begin{theorem}{(\cite[Theorem 3]{Engwerda:12},\cite[Theorem 5.2]{Reddy:13})}
	\label{thm:nesscond}
	Assume $\bar R_{ii}\succ0$ for all $i\in \mathsf N$. 
	Then the strategy profile $F^\star:=\col\{ F_i^\star\}_{i\in \mathsf N}\in \mathsf F_s$ 
	is an FBNE for every consistent initial state if and only if
	\begin{subequations}
		\begin{align}
			&F^\star \in \Omega^{-1}(\bar F^\star), ~\bar F^\star \in \mathcal F_{\mathrm{red}},\label{eq:fullstatefb}\\ 
			&\mathcal F_\mathrm{red}:=\Big\{\bar F^\star \in \mathbb R^{m\times r}~|~\bar G \bar F^\star =  - (\bar V^\top+\bar B_d^\top \bar P)\Big\}, \label{eq:FBNEeq1}
		\end{align}
		where $\bar V=\row\{\bar V_{ii}\}_{i=1}^N$, $\bar B_d=\blkdiag\{\bar B_{1i}\}_{i=1}^N$, 
		$\bar G=[\bar G_{ij}]_{i,j\in\mathsf N}$ with
\begin{align} 
		\bar G_{ij}=
		\begin{cases}
			\bar R_{ij}, & i=j,\\
			\bar S_{iij}, & i\neq j,
		\end{cases}\label{eq:FBNEeq2} 
\end{align} 
		and $\bar P=\col\{\bar P_i\}_{i=1}^N$ are symmetric solutions of the coupled algebraic Riccati equations (CARE)
		\begin{align}
			\bar A_\mathrm{cl}^\top \bar P_i +\bar P_i \bar A_\mathrm{cl}
			+
			\begin{bmatrix} \I_r\\ \bar F^\star\end{bmatrix}^\top
			\bar M_i
			\begin{bmatrix} \I_r\\ \bar F^\star\end{bmatrix}
			=0,\quad i\in \mathsf N,
			\label{eq:CARE}
		\end{align}
		\label{eq:FBNENScond}%
			\end{subequations}
		satisfying $\sigma(\bar A_\mathrm{cl})\subset \mathbb C^-$, where $\bar A_\mathrm{cl}=J+\bar B_1\bar F^\star$. Further, the equilibrium cost of Player $i\in \mathsf N$ is given by $x_1^\top(0)\bar P_i x_1(0).$
	\end{theorem}
\begin{remark}
	\label{rem:infononuniqueness}
	For each $\bar F^\star \in \mathcal F_\mathrm{red}$, the associated full-state FBNE of the LQDDG are precisely the elements of the preimage $\Omega^{-1}(\bar F^\star)\subset \mathsf F_s$. By Lemma~\ref{lem:highlowtransformation}, distinct equilibria in $\Omega^{-1}(\bar F^\star)$ correspond to the same reduced feedback law $\bar F^\star$ and induce identical closed-loop behavior and equilibrium costs. Hence, the equilibria in $\Omega^{-1}(\bar F^\star)$ are referred to as informationally non-unique; see also~\cite{Engwerda:12,Reddy:13}.
\end{remark}
 \subsection{Solution to the inverse problem}  
 %%%
The following theorem gives the main result by characterizing the solution set~\eqref{eq:solset1} via a reformulation of~\eqref{eq:CARE} and~\eqref{eq:FBNEeq1}.
%%%
%%%
\begin{theorem} \label{thm:invthm1}  
 	\tb{Consider the data $E$, $A$, and $B_i$, $i\in\mathsf N$, specified in Problem~\ref{prob:prob1}, and let the observed strategy profile $F\in\mathsf F_s$ be a stabilizing feedback profile for the state dynamics~\eqref{eq:state_eqn}}.
 	For each player $i\in\mathsf N$, define the sets
 	\begin{subequations} 
 		\begin{align}
 			\Gamma^1_i 
 			&:=\Big\{\{Q_i,R_{ij}\}_{i,j\in\mathsf N}~\big|~   \mathcal M^i_Q\operatorname{vec}(Q_i)
 			+\sum_{j\in \mathsf N} \mathcal M^{ij}_R\operatorname{vec}(R_{ij})=0\Big\},
 			 			\label{eq:G1set} \\
 			\Gamma^2_i 
 			&:=\Big\{\{Q_i,R_{ij}\}_{i,j\in\mathsf N}~\big|~   R_{ii}+\bar B_{2i}^\top X_2^\top Q_i X_2 \bar B_{2i}\succ 0 \Big\}, 
 			\label{eq:G2set} 
 		\end{align}
 		where  $\mathcal M_Q^i\in \mathbb R^{rm_i\times n^2}$ and $\mathcal M_R^{ij}\in \mathbb R^{rm_i\times m_j^2}$ are given by
 		\begin{equation} 
 			\label{eq:terms}
 		\begin{aligned}
 			\mathcal M_Q^i&:=\mathcal N_Q^i-(\I_r\otimes \bar B_{1i}^\top)\mathcal K^{-1}\mathcal M_Q,\\
 			\mathcal M_R^{ij}&:=\begin{cases}  (\bar F_i^\top \otimes \I_{m_i})-(\I_r\otimes \bar B_{1i}^\top)\mathcal K^{-1}(\bar F_j^\top\otimes \bar F_j^\top),&i=j\\
 				-(\I_r\otimes \bar B_{1i}^\top)\mathcal K^{-1}(\bar F_j^\top\otimes \bar F_j^\top),&i\neq j
 			 \end{cases} \\
 		 \bar F&=\Omega(F)\in \mathbb R^{m\times r},\\
 		 	\mathcal K&:=(\I_r\otimes \bar A_\mathrm{cl}^\top)+( \bar A_\mathrm{cl}^\top\otimes \I_r)\in \mathbb R^{r^2\times r^2},\\
 		 \mathcal M_{Q}&:=X_1^\top \otimes X_1^\top-\bar F^\top \bar B_2^\top X_2^\top \otimes X_1^\top-X_1^\top \otimes \bar F^\top \bar B_2^\top X_2^\top\notag \\
 		 &\quad  +\bar F^\top \bar B_2^\top X_2^\top \otimes \bar F^\top \bar  B_2^\top X_2^\top\in \mathbb R^{r^2\times n^2},\\
 		 	\mathcal N_Q^i&:=\bar F^\top \bar B_2^\top X_2^\top \otimes \bar B_{2i}^\top X_2^\top -X_1^\top \otimes \bar B_{2i}^\top X_2^\top \in \mathbb R^{rm_i\times n^2}.
 		 \end{aligned} 
 	 \end{equation} 
 	\tb{Then, the solution  set~\eqref{eq:solset1} is given by}
 		\begin{align}
 			\Theta(F)
 			=
 			\prod_{i\in\mathsf N}(\Gamma_i^1\cap\Gamma_i^2).
 			\label{eq:Thetaset}
 		\end{align}
 	\end{subequations} 
 \end{theorem} 
 \begin{proof}
 	Let $\{Q_i,R_{ij}\}_{i,j\in \mathsf N}$ be the cost parameters for which the given strategy profile $F\in \mathsf F_s$
 	is an FBNE. Then the necessary and sufficient conditions \eqref{eq:FBNENScond} stated in Theorem~\ref{thm:nesscond} hold. 
 	The CARE \eqref{eq:CARE} can be expanded as
 	\begin{align*} 
 		&  \bar A_\mathrm{cl}^\top \bar P_i +\bar P_i \bar A_\mathrm{cl}
 		+ X_1^\top Q_i X_1 - X_1^\top Q_i X_2 \bar B_2 \bar F \notag\\
 		& - \bar F^\top \bar B_2^\top X_2^\top Q_i X_1
 		+ \bar F^\top \bar B_2^\top X_2^\top Q_i X_2 \bar B_2 \bar F
 		+ \bar F^\top R_i \bar F =0 .
 	\end{align*}
 	Using the product formula~\cite{Brewer:1978} for vectorization,
 	$\mathrm{vec}(XYZ)=(Z^\top\otimes X)\mathrm{vec}(Y)$,
 	for matrices $X$, $Y$, and $Z$ of appropriate dimensions, and applying it term by term, we obtain
 	\begin{align}
 		\mathcal K \vecop(P_i)+\mathcal M_Q \vecop(Q_i)
 		+\sum_{j\in \mathsf N}(\bar F_j^\top\otimes \bar F_j^\top)\vecop(R_{ij}) =0 .
 		\label{eq:CAREcond}
 	\end{align}
 
 	Expanding condition \eqref{eq:FBNEeq1} for player $i\in \mathsf N$ yields
 	\begin{align*}
 		\sum_{j\in \mathsf N} \bar G_{ij}\bar F_j
 		= \bar B_{2i}^\top X_2^\top Q_i X_1 -\bar B_{1i}^\top P_i .
 	\end{align*}
 	Substituting for $\bar G_{ij}$ using \eqref{eq:FBNEeq2}, we obtain
 	\[
 	R_{ii}\bar F_i
 	+ \sum_{j\in \mathsf N} \bar B_{2i}^\top X_2^\top Q_i X_2 \bar B_{2j}\bar F_j
 	-\bar B_{2i}^\top X_2^\top Q_i X_1
 	+\bar B_{1i}^\top P_i =0 .
 	\]
 	Vectorizing the above expression gives
 	\begin{align}
 		&(\I_r\otimes \bar B_{1i}^\top)\vecop(P_i)
 		+\mathcal N^i_{Q}\vecop(Q_i) %\notag\\&\qquad 
 		+(\bar F_i^\top\otimes \I_{m_i})\vecop(R_{ii})=0 .
 		\label{eq:FBNEinvcond}
 	\end{align}
 	Since $\bar P$ is a stabilizing solution of the CARE \eqref{eq:CARE}, the matrix $\bar A_\mathrm{cl}$ is stable. 
 	Hence, $\bar A_\mathrm{cl}$ and $-\bar A_\mathrm{cl}$ have no eigenvalues in common, implying that the matrix
 	$
 	\mathcal K=(\I_r\otimes\bar A_\mathrm{cl}^\top)+(\bar A_\mathrm{cl}\otimes \I_r)$
 	is invertible. Using this together with \eqref{eq:CAREcond} and \eqref{eq:FBNEinvcond}, we obtain the set \eqref{eq:G1set}, which characterizes all cost parameters satisfying \eqref{eq:FBNEeq1} and \eqref{eq:CARE}. 
 	Further, the set \eqref{eq:G2set} collects all cost parameters satisfying $\bar R_{ii}\succ 0$ for all $i\in \mathsf N$. 
 	Consequently, the set \eqref{eq:Thetaset} characterizes the cost parameters for which the observed strategy profile $F\in \mathsf F_s$ is an FBNE.
 \end{proof}
%%%%%%%%%%%%%%%%%%%%%%%%%%%%%% 
%%%%%%%%%%%%%%%%%%%%%%%%%%%%%%
Being a Cartesian product, the solution set~\eqref{eq:Thetaset} is rectangular, which allows the cost parameters to be identified in a distributed manner. Since the cost parameters are symmetric, we use the half-vectorization operator and define the per-player cost parameters as
%%%%
%%%%
%%%%
\begin{align*} 
	\theta_i :=[\vech(Q_i)^\top \vech(\col\{R_{ij}\}_{j\in \mathsf N})^\top]^\top \in \mathbb R^{L},
\end{align*}
where $L=\tfrac{1}{2}\big( n(n+1) +\sum_{j=1}^N m_j (m_j+1) \big)$. 
Using this, \eqref{eq:G1set} and \eqref{eq:G2set} are equivalently represented by
\begin{subequations} 
	\begin{align}
		\Gamma_i^1 
		&:= \big\{ \theta_i \in \mathbb R^{L} \,\mid\, \mathcal M_i \theta_i = 0 \big\}, \label{eq:Gamma_a}
		\\
		\Gamma_i^2 
		&:=\big\{ \theta_i \in \mathbb R^{L} \,\mid\,
		R_{ii}+\bar B_{2i}^\top X_2^\top Q_i X_2 \bar B_{2i}\succ 0 \big\},
		\label{eq:Gamma_b}\\
		\mathcal M_i
		&:= \begin{bmatrix} 
			\mathcal M_Q^i \mathcal D_n &
			\row\{\mathcal M_R^{ij} \mathcal D_{m_j}\}_{j=1}^N 
		\end{bmatrix} \in \mathbb R^{rm_i\times L}. \label{eq:Gamma_c}
	\end{align}
\end{subequations} 
\begin{corollary}
	\label{cor:QP}
	For each $i\in \mathsf N$, consider the optimization problem
	\begin{align}
		\min_{\theta_i \in \Gamma_i^2} f_i(\theta_i),\quad 
		~ f_i(\theta_i):=\tfrac{1}{2}\|\mathcal M_i \theta_i\|_2^2 .
		\label{eq:opt1prob}
	\end{align}
	Then, $
	\arg\min_{\theta_i\in\Gamma_i^2} f_i(\theta_i)
	=
	\Gamma_i^1\cap \Gamma_i^2$.
	In particular, the solution set of the inverse problem~\eqref{eq:solset1} is non-empty if and only if the optimal value of \eqref{eq:opt1prob} equals zero for every $i\in\mathsf N$.
\end{corollary}
\begin{proof}
	Observe that $f_i(\theta_i)=0$ if and only if $\mathcal M_i\theta_i=0$, which is equivalent to $\theta_i\in\Gamma_i^1$. 
	Hence, minimizing $f_i$ over $\Gamma_i^2$ yields   the set $\Gamma_i^1\cap\Gamma_i^2$.
	The non-emptiness of the solution set  \eqref{eq:solset1} follows from the Cartesian product structure in \eqref{eq:Thetaset}.
\end{proof}
\begin{remark} \label{eq:QPalg}
A consequence of Corollary~\ref{cor:QP} is that the non-emptiness of~\eqref{eq:solset1} can be verified by solving the convex quadratic program~\eqref{eq:opt1prob}, as in inverse LQDGs~\cite{Inga2019CSL}. An optimal value of zero certifies that the solution set~\eqref{eq:Thetaset} is nonempty. \tb{Thus, $\|\mathcal M_i\theta_i\|_2$ measures identification error for player~$i$, and exact identification is achieved when $\|\mathcal M_i\theta_i\|_2=0$ for all $i\in\mathsf N$.}
\end{remark}
%%%%%%
\begin{remark}\label{rem:scaling} For every $\theta_i^\star \in \Gamma_i^1\cap \Gamma_i^2$, we have $\kappa_i \theta_i^\star \in \Gamma_i^1\cap \Gamma_i^2$ for any $\kappa_i>0$. 
	Hence, infinitely many cost parameters can rationalize the observed strategies as an FBNE, reflecting the ill-posedness of inverse differential game problems, similar to inverse LQDG \cite[Corollary 1]{Inga2019CSL}.
\end{remark}
 
\subsection{Comparison with inverse LQDG \cite{Inga2019CSL}}
In this subsection, we highlight key differences between the solution set \eqref{eq:Thetaset}
and that of inverse LQDG \cite{Inga2019CSL} (i.e., when $E$ is nonsingular).
These differences appear in three aspects.

\subsubsection{Cost parameters} 
In inverse LQDG \cite{Inga2019CSL}, the cost parameters $R_{ii}$, $i\in \mathsf N$, are required to be positive definite. In contrast, inverse LQDDG imposes the definiteness condition jointly on the state and control weighting matrices, as specified in \eqref{eq:G2set}.

\subsubsection{Dependence on $\mathrm{rank}(E)$}   
\tb{The condition $\mathrm{rank}(E)=r$ determines the dimension of the constraint manifold induced by the transformation~\eqref{eq:transformation}, on which the reduced state evolves according to~\eqref{eq:reddyn}. For inverse LQDGs~\cite{Inga2019CSL}, the dimension of the per-player solution set~\eqref{eq:G1set} has lower bound $L-nm_i$. For inverse LQDDGs, this lower bound increases as the dimension of the constraint manifold decreases, as shown next.
\begin{corollary}\label{cor:rank}
	Let $\mathrm{rank}(E)=r<n$. Then
	\begin{align}
		\dim(\Gamma_i^1) \ge L-rm_i > L-nm_i > 0 .
	\end{align}
\end{corollary}
\begin{proof}
	From the rank–nullity theorem we have $
	\dim(\Gamma_i^1)
	=\dim(\ker(\mathcal M_i))
	= L-\mathrm{rank}(\mathcal M_i)
	\ge L-rm_i$. 
	Since $r<n$, it follows that $L-rm_i > L-nm_i$. Moreover, $
	L-nm_i
	=\tfrac{1}{2}(n+m_i)+\tfrac{1}{2}\sum_{j\neq i}m_j(m_j+1)
	+\tfrac{1}{2}(n-m_i)^2>0$.
\end{proof}} 
\tb{Corollary~\ref{cor:rank} shows that, in the descriptor setting, \(\Gamma_i^1=\ker(\mathcal M_i)\) has a larger lower bound on its dimension than in the inverse LQDG case \cite{Inga2019CSL}. Thus, the descriptor structure introduces additional degrees of freedom in the inverse map from observed equilibrium behavior to admissible cost parameters, allowing more cost realizations to rationalize the same observed FBNE. This reflects reduced identifiability induced by the algebraic constraints in the descriptor dynamics.}
\subsubsection{Forward problem with identified costs}   \label{sec:forwardproblem} 
\tb{ 
It is well known that a differential game may admit multiple FBNE due to the non-uniqueness of Nash equilibria \cite{Basar:1999,Engwerda:2005}. Consequently, identified costs may rationalize behaviors other than the one induced by the observed feedback strategy profile. In inverse LQDGs \cite{Inga2019CSL}, the number of rationalized behaviors equals the number of stabilizing solutions of the associated CARE \cite{Engwerda:2005}. In inverse LQDDGs, however, the interpretation is more subtle because multiple FBNE may correspond to the same reduced feedback and thus to the same closed-loop behavior. In particular,
\begin{align*}
	\mathcal F(\{Q_i,R_{ij}\}_{i,j\in\mathsf N})
	=
	\bigcup_{\bar F^\star \in \mathcal F_\mathrm{red}}
	\Omega^{-1}(\bar F^\star),
\end{align*}
where $\mathcal F_\mathrm{red}$ is obtained from the stabilizing solutions of \eqref{eq:CARE} via \eqref{eq:FBNEeq1}. Each pre-image $\Omega^{-1}(\bar F^\star)$, $\bar F^\star\in\mathcal F_\mathrm{red}$, may contain uncountably many FBNE, all inducing the same closed-loop behavior. Hence, the identified costs may rationalize uncountably many FBNE but only $|\mathcal F_\mathrm{red}|$ distinct behaviors.}

%%%%%%%%%%%%%
\section{Numerical illustration} 
\label{sec:Numerical} 
\tb{We consider a two-player shared-steering lane-keeping example in which a human driver~($h$) and an automation unit~($a$) simultaneously apply steering torques through the same steering interface. This example illustrates inverse-game identification in shared-steering control; see \cite{IngaECC2021,LinWu2024}. Let $x=[e_y,\;e_\psi,\;\delta]^\top\in\mathbb{R}^3$ denote the lateral lane error~(m), heading error~(rad), and steering angle~(rad), with scalar inputs $u_h$ and $u_a$~(Nm). The descriptor structure follows from the kinematic lane-error dynamics $\dot e_y=v_x e_\psi$, $\dot e_\psi=(v_x/L)\delta$, and the quasi-static steering constraint $0=-K_s\delta+u_h+u_a$, where $L>0$ and $K_s>0$, yielding~\eqref{eq:state_eqn} with
\begin{align*}
	E &= \mathrm{diag}(1,1,0), \quad
	A = \begin{mymatrix}
		0 & v_x & 0\\
		0 & 0 & v_x/L\\
		0 & 0 & -K_s
	\end{mymatrix}, \quad
	B_h=B_a=\begin{mymatrix}0\\0\\1\end{mymatrix}.
\end{align*}
The DAE is index-1, with algebraic constraint $\delta=(u_h+u_a)/K_s$. Each player $i\in\{h,a\}$ minimizes
\begin{equation*}
	J_i=\int_0^\infty \!\big(
	x^\top Q_i x+\sum_{j\in\{h,a\}} R_{ij}u_j^2
	\big)\,dt,
\end{equation*}
where $Q_i=Q_i^\top$ and $R_{ij}\in\mathbb{R}$.
We simulate with $v_x=20$~m/s, $L=2.7$~m, and $K_s=10$~Nm/rad. The ground-truth costs are
\begin{align*} 
&Q_h^{\mathrm{gt}} = \mathrm{diag}(1.0,\;0.5,\;0.1), ~Q_a^{\mathrm{gt}} = \mathrm{diag}(3.0,\;2.0,\;0.1),\\
&R_{hh}^{\mathrm{gt}} = 2.0,~ R_{ha}^{\mathrm{gt}} = 0.5,~ R_{ah}^{\mathrm{gt}} = 1.0,~
R_{aa}^{\mathrm{gt}} = 0.5.
\end{align*} 
Thus, the automation penalizes lane-keeping and heading errors more strongly than the human, whereas the human is more effort-averse. The CARE~\eqref{eq:CARE} admits a unique symmetric stabilizing solution with reduced feedback matrix
$
\bar{F}_{\mathrm{gt}}=
\begin{mymatrix}
	-1.9995 & -0.3547\\
	-9.5252 & -2.1631
\end{mymatrix},
$
and the corresponding FBNE trajectories are shown in Fig.~\ref{fig:figa}. Treating these as the observed   state and control trajectories, least-squares fitting (Remark~\ref{rem:feedbackform}) yields the observed feedback profile
$
F_{\mathrm{gt}}=
\begin{mymatrix}
	-0.0046 & -0.3971 & 0.7987\\
	-0.4779 & -1.8117 & 3.8449
\end{mymatrix}
\in \Omega^{-1}(\bar F_\mathrm{gt}),
$
which is used for cost identification in the inverse problem. We use CVX~\cite{CVX} to solve the convex optimization problem~\eqref{eq:opt1prob}.}
%%%%%%%%%%%%%%%%%%%%%%%%%%%%%%%%%%%%%%%%%%%%%%%%

%%%%%%%%%%%%%%%%%%%%%%%%%%%%%%%%%%%%%%%%%%%%%%%%
\subsubsection*{Identification}
\tb{Solving the inverse game~\eqref{eq:opt1prob} with ${F}_{\mathrm{gt}}$ as the observed FBNE yields one identified cost tuple
\begin{align*}
	Q_h^{\mathrm{id}}&={\footnotesize\begin{bmatrix}
			0.106 & 0.017 & 0.050\\
			0.017 & -0.328 & 1.810\\
			0.050 & 1.810 & -0.621
	\end{bmatrix}},
	Q_a^{\mathrm{id}}={\footnotesize\begin{bmatrix}
			1.378 & 1.764 & -0.499\\
			1.764 & 0.985 & 0.035\\
			-0.499 & 0.035 & -0.752
	\end{bmatrix}}\\
	R_{hh}^{\mathrm{id}}&=0.564,\quad
	R^{\mathrm{id}}_{ha}=0.152,\quad
	R_{ah}^{\mathrm{id}}=-0.709,\quad
	R^{\mathrm{id}}_{aa}=0.231.
\end{align*}
The forward descriptor game with these costs admits two stabilizing solutions of the CARE~\eqref{eq:CARE}, with reduced feedback set
$
\mathcal F^{\mathrm{id}}_\mathrm{red}=
\left\{\bar F^{(1)}_{\mathrm{id}}, \bar F^{(2)}_{\mathrm{id}}\right\}$,
where
$
\bar F^{(1)}_{\mathrm{id}}=
\begin{mymatrix}
	-7.219 & -0.3547\\
	2.5253 & -2.1631
\end{mymatrix}$,
$\bar F^{(2)}_{\mathrm{id}}=
\begin{mymatrix}
	-1.9995 & -0.3547\\
	-9.5252 & -2.1631
\end{mymatrix}$,
and the FBNE are given by
$$
\Omega^{-1}(\bar F^{(k)}_\mathrm{id})
=
\left\{
\begin{bmatrix}
	ax_1+b & c_kx_1+d_k & x_1\\
	ax_2+e & c_kx_2+f_k & x_2
\end{bmatrix}
:\;
{\small
	\begin{array}{l}
		x_1,x_2\in\mathbb R,\\
		x_1+x_2\neq 10
\end{array}}
\right\},
$$
where $a=0.4383$, $b=-0.3547$, $e=-2.163$, $c_1=0.8171$, $c_2=2.006$, $d_1=-7.219$, $d_2=-1.999$, $f_1=2.525$, and $f_2=-9.525$.
Thus, the identified costs rationalize two distinct behaviors, as discussed in Subsection~\ref{sec:forwardproblem}.
In particular, all informationally non-unique FBNE in $\Omega^{-1}(\bar F^{(2)}_{\mathrm{id}})$ reproduce the same behavior as the observed equilibrium since $\bar F_\mathrm{id}^{(2)}=\bar F_{\mathrm{gt}}$ (see Fig.~\ref{fig:figa}), whereas the FBNE in $\Omega^{-1}(\bar F^{(1)}_{\mathrm{id}})$ induce behavior different from the observed one (see Fig.~\ref{fig:figb}).}
%%%%%%%%%%%%%%%%%%%%%%%%%%%%%%%%%%%%
%%%%%%%%%%%%%%%%%%%%%%%%%%%%%%%%%%%%
\subsubsection*{Identification with  diagonal cost matrices} 
\tb{We solve~\eqref{eq:opt1prob} under a diagonal structural constraint, obtaining
$Q_h^{\mathrm{diag}}=\mathrm{diag}(0.308,~0.801,~-0.991)$,
$R_{hh}^{\mathrm{diag}}=0.967$, $R_{ha}^{\mathrm{diag}}=0.237$,
$Q_a^{\mathrm{diag}}=\mathrm{diag}(0.152,~-1.355,~0.801)$,
$R_{ah}^{\mathrm{diag}}=0.340$, and $R_{aa}^{\mathrm{diag}}=0.007$.
The forward game admits a unique stabilizing solution to the CARE~\eqref{eq:CARE} whose reduced feedback matrix \eqref{eq:FBNEeq1} matches $\bar{F}_\mathrm{gt}$. Hence, these costs lie in $\Theta(F_\mathrm{gt})$ and rationalize the observed strategies as an FBNE. Although structurally similar to the ground-truth costs, they are not positive scalings of the true parameters (see Remark~\ref{rem:scaling}). Thus, structural constraints shrink the solution set but do not ensure recovery of the ground truth, illustrating the ill-posedness of the inverse problem.}
\begin{figure}[h]  	\centering
	\subfloat[\tb{Observed; reproduced by $\Omega^{-1}(\bar F^{(2)}_{\mathrm{id}})$}] {%
		{\includegraphics[scale=.94]{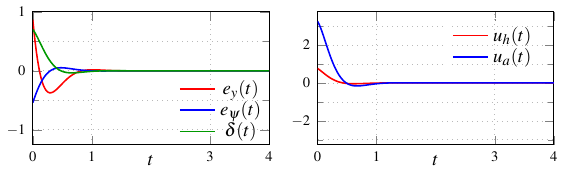}} \label{fig:figa}}\\
	\subfloat[\tb{generated by $\Omega^{-1}(\bar F^{(1)}_{\mathrm{id}})$}]{%
		{\includegraphics[scale=.94]{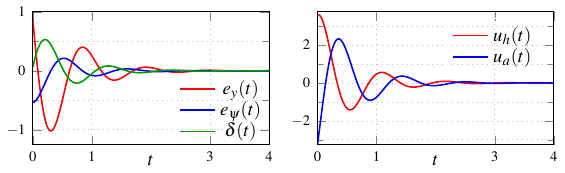}} \label{fig:figb}}\\ 
	\subfloat[\tb{generated by  $\Omega^{-1}(\bar F^{(k)}_{\mathrm{mis}})$, $k=1,2,3,4$, and $\Omega^{-1}(\bar F^{(2)}_{\mathrm{id}})$}]{\includegraphics[scale=.85]{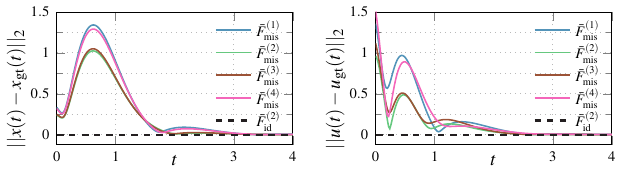}\label{fig:figc}}
	\caption{{\tb{\small Panels (a) and (b) illustrate state (left) and control (right)  trajectories. Panel (c) illustrates state error (left) and control error (right) trajectories due to model misspecification.}}}
\end{figure}
\subsubsection*{Misspecification}
\tb{To quantify the effect of ignoring the descriptor structure, we apply the inverse LQDG method of~\cite{Inga2019CSL} with $E=I_3$, thereby replacing the descriptor interaction model by a standard ODE model and treating the instantaneous steering constraint $\delta=(u_h+u_a)/K_s$ as a dynamic equation. This yields
$\theta_1^{\mathrm{mis}}=[0.1161~ 0.2721~ -1.3731~ 0.6642~ 1.0042~ 1.0853~ 0.4184~ -0.5001]^\top$
and
$\theta_2^{\mathrm{mis}}=[0.0119~ 0.2462~ -0.1107~ -0.7575~ 0.0731~ -3.1384~ -1.6787~ 0.0508]^\top$.
However, the LQDDG null-space conditions~\eqref{eq:Gamma_a} give
$\|\mathcal{M}_1\theta_1^{\mathrm{mis}}\|=3.31\neq 0$ and
$\|\mathcal{M}_2\theta_2^{\mathrm{mis}}\|=0.28\neq 0$,
so $(\theta_1^{\mathrm{mis}},\theta_2^{\mathrm{mis}})\notin\Theta(F_\mathrm{gt})$, confirming the failure in identification; see Remark~\ref{eq:QPalg}. The corresponding forward descriptor game admits four stabilizing solutions of the CARE~\eqref{eq:CARE}, with reduced feedback matrices
\begin{align*}
	\bar F_\mathrm{mis}^{(1)}&=\begin{mymatrix}0.5366 & 0.1129\\ -5.1587 & -0.5605\end{mymatrix},\quad
	\bar F_\mathrm{mis}^{(2)}=\begin{mymatrix}-0.3965 & -0.1167\\ -5.3210 & -0.3660\end{mymatrix},\\
	\bar F_\mathrm{mis}^{(3)}&=\begin{mymatrix}0.3495 & -0.1167\\ -5.9530 & -0.3660\end{mymatrix},\quad
	\bar F_\mathrm{mis}^{(4)}=\begin{mymatrix}-0.5757 & 0.1129\\ -4.2164 & -0.5605\end{mymatrix}.
\end{align*}
None of the FBNE in $\Omega^{-1}(\bar F^{(k)}_\mathrm{mis})$, $k=1,2,3,4$, reproduce the observed behavior. Fig.~\ref{fig:figc} shows the resulting state error ($||x(t)-x_\mathrm{gt}(t)||_2$) and control error ($||u(t)-u_\mathrm{gt}(t)||_2$) trajectories relative to the ground truth. This misspecification replaces the instantaneous steering constraint with an artificial first-order steering dynamics, thereby attributing constraint-induced behavior to strategic preferences, as discussed in Section~I.}

\section{Conclusions} 
\label{sec:Conclusions}
 \tb{In this paper, we formulated and solved an inverse problem for infinite-horizon linear--quadratic descriptor differential games. Given an observed feedback strategy profile, we characterized all cost parameters that rationalize it as a feedback Nash equilibrium. We showed that the resulting solution set is rectangular and convex, and proposed a convex optimization program for computing an admissible realization. Compared with the standard inverse LQDG setting, we showed that descriptor dynamics may enlarge the solution set, thereby reducing identifiability. A natural direction for future work is to extend the framework to a model-free setting.}

\bibliographystyle{IEEEtran}
\bibliography{references_DAE_Final} 
\end{document}